\documentclass[a4paper,12pt]{amsart}
\usepackage[T1]{fontenc}    %sajat
\usepackage{enumerate}
\usepackage{amssymb}
\usepackage{tikz, subfigure}
\usetikzlibrary{patterns}

\def\RR{\mathbb R}

\def\nn{\mathbb N}

\newcommand{\set}[1]{\left\lbrace #1\right\rbrace}%set

\newcommand{\remove}[1]{ }
\newcommand{\qtq}[1]{\quad\text{#1}\quad}%\quad\text{and}\quad

\newtheorem{theorem}{Theorem}[section]
\newtheorem{proposition}[theorem]{Proposition}%
\newtheorem{lemma}[theorem]{Lemma}

\theoremstyle{definition}

\theoremstyle{remark}

\remove{}

\numberwithin{equation}{section}
\begin{document}
\title[]{Visibility of non-self-similar sets  }
\author{Yi Cai*, Yang Yang}
\address{School of Sciences, Shanghai Institute of Technology, Shanghai 201418,
People's Republic of China}
\email{caiyi@sit.edu.cn}
\address{School of Sciences, Shanghai Institute of Technology, Shanghai 201418,
People's Republic of China}
\email{2847276260@qq.com}
%\keywords{Visible Part; Homogeneous Cantor Sets; Hausdorff Dimension.}
%\curraddr{}
%\date{Version 2019-07-31}
%\dedicatory{}
\thanks{*Corresponding author}

\begin{abstract}
%% Text of abstract
%Zhang, Jiang and Li (2020) considered the visibility of self-similar set and some
%descriptions of visible set have been given. In our work, we further observe the visibility problems on non-self-similar sets, in this case the structure becomes more complicated.
The visible problem is
related to the arithmetic on the fractals. The visibility of self-similar set has been studied in the past. In this work, we investigate the visibility of non-self-similar sets. We begin by analyzing the structure of $F^2_\lambda$, where $F^2_{\lambda}:=\set{x^2:x\in F_{\lambda}}$ and $F_{\lambda}$ is the middle $1-2\lambda$ Cantor set, we show that it lacks self-similarity. Due to the nonlinear phenomena exhibited by $F^2_\lambda$, we develop a different approach to characterize the visible set.
%combining methods from fractal theory, numerical computation, and dynamical systems theory.
Our results also reveal that the visible set may contain a closed interval within a large range of $\lambda$.% and its structure differs from that of self-similar sets.

\end{abstract}
\maketitle
\section{Intorduction}

Suppose $0<\lambda<1$ and $F_{\lambda}$ is the attractor generated by
\begin{equation*}
 \left\{f_{n}(x)=\lambda x+n(1-\lambda)\right\}_{n=0}^{1}.
\end{equation*}
%and $C_{\lambda}=f_{1}\left(K_{\lambda}\right) \cup f_{2}\left(K_{\lambda}\right)$.
We note that if $\frac{1}{2}\le \lambda<1$, at this time, $F_{\lambda}=[0,1]$; if $0<\lambda<\frac{1}{2}$ then it is what we commonly know as the middle $1-2\lambda$ Cantor set, from now on we always assume that the second condition is satisfied.       % or middle $1-2\lambda$ Cantor set.

There have been abundant results regarding Cantor sets. Athreya, Reznick and Tyson \cite{ART} proved every real number between zero and one is the product of three elements of the middle third Cantor set. Baker and Kong \cite{BK} studied the structures and dimension of $F_{\lambda}\cap (F_{\lambda}+a)$ for a real number $a$. Pourbarat determined all triples $(\lambda_1,\lambda_2,a)$ such that $F_{\lambda_1}- aF_{\lambda_2}$ forms a closed interval \cite{PM}. Furthermore, he \cite{PM1} gave five possible structures for the arithmetic sum of homogeneous Cantor sets.

The problems of visibility arise from the non-trivial extension of arithmetic operations to fractal sets. Ren et al.\cite{RJ} provided a sufficient condition such that the image set of the Cartesian product of Moran set under the continuous function contains interior points. Tian et al.\cite{TJ} showed a necessary and sufficient condition such that the algebraic product of the self-similar sets with overlaps is equal to the interval $[0,1]$. For more details, we refer the readers to \cite{JJSW,OT,Re,BB}.
%\begin{equation*}
%  \{(x,f(x)):x\in D\}\cap A=\emptyset.
%\end{equation*}

Since middle Cantor sets have many elegant properties, such as self-similarity, Zhang, Jiang and Li \cite{ZJL} considered the visibility of $F_{\lambda}\times F_{\lambda}$ with respect to the  lines, in this case $F_{\lambda}\times F_{\lambda}$ is self-similar. Motivated by their work, it is natural to observe the non-self-similar case. However, the structure becomes more complicated.

Let
\begin{equation}\label{D-1}
F^2_{\lambda}:=\set{x^2:x\in F_{\lambda}}\qtq{and}D:=\left\{\frac{x_1^2}{x_2}: x_1, x_2 \in F_{\lambda}\;\text{and}\; x_2 \neq 0\right\}.
\end{equation}
We call a curve $y=f(x),x\in \mathcal D\subset \mathbb R$ \emph{visible} with respect to a plane set $S\subset \mathbb R^2$ if $ \{(x,f(x)):x\in \mathcal D\}\cap S=\emptyset$. In our work, we are interested in the \emph{visibility} of set $F_{\lambda}\times F^2_{\lambda}$, i.e., we consider the lines $y=\alpha x,\alpha\ge 0, x\in \mathbb R^+$ which are visible with respect to $F_{\lambda}\times F^2_{\lambda}$. The research object is the visible set
\begin{equation*}
K:=\set{k \geq 0:k \notin D }.
\end{equation*}
%\begin{equation*}
 %K:=\left\{k \geq 0: f(x)=k x \; \text{is visible  with respect to $(C_{\lambda}\backslash\{0\}) \times C^2_{\lambda}$}\right\}
%\end{equation*}
%where $C^2_{\lambda}:=\set{x^2:x\in C_{\lambda}}$. It means that
%\begin{equation*}
%k \notin D:=\left\{\frac{x_1^2}{x_2}: x_1, x_2 \in C_{\lambda}\;\text{and}\; x_2 \neq 0\right\}.
%\end{equation*}
%and $
%V=[0,+\infty) \backslash \frac{C^2_{\lambda}}{C_{\lambda} \backslash\{0\}}$.

%By $A^{o}, m(A)$ and $\operatorname{dim}_{H}(A)$ we denote the set of interior points, the Lebesgue measure and the Hausdorff dimension of $A$, respectively. Now, we state our results.
%Now we state our main result:
\begin{theorem}\label{theorem-1} Let $F^2_{\lambda}$ and $D$ be defined in \eqref{D-1}. Then
\begin{enumerate}[\upshape (i)]

\item The set $F^2_{\lambda}$ is non-self-similar.

\item If $\frac{1}{3} \leq \lambda < \frac{1}{2}$, then the visible set $K=\emptyset$ .
\item If $\frac{2-\sqrt{2}}{2}\le \lambda <\frac{1}{2}$, then $D$ contains a closed interval.
%Moreover,  $\operatorname{dim}_{H}(\frac{C^2_{\lambda}}{C_{\lambda} \backslash\{0\}})=1$ when $\frac{1}{4}\le \lambda < \frac{1}{2}$ and $\operatorname{dim}_{H}(\frac{C^2_{\lambda}}{C_{\lambda} \backslash\{0\}})=\frac{\log 4}{-\log \lambda}$ when $0<\lambda < \frac{1}{4}$.

\end{enumerate}
\end{theorem}

  In the rest of paper we describe the set $F^2_{\lambda}$ in Sec.\ref{S2}, the main proofs of Theorem 1.1 are given in Sec.\ref{S3}. In the finally section, we give some remarks.

\section{Description of $F^2_{\lambda}$}\label{S2}
%As is well known, $F_\lambda$ is constructed by removing open intervals $(\alpha_i,\beta_i)$ at each step, where the total removed length maintains a fixed proportion. Specifically, up to the $k$-th construction step, we remove $\sum_{n=0}^{k-1}2^n$ open intervals, yielding the representation $F_\lambda=[0,1]\setminus \cup(\alpha_i,\beta_i)$. $F^2_{\lambda}$ can be similarly constructed using an analogous approach, the only change is that the proportion of the deleted intervals is no longer a fixed value.

As is well known, $F_\lambda$ is constructed by removing the intervals
\begin{align*}
&(\lambda,1-\lambda)\quad \text{at first step;}\\
%&(\lambda^2,\lambda(1-\lambda)),\quad(\lambda^2+1-\lambda,1-\lambda^2)\quad \text{at second step}\\
&f_0((\lambda,1-\lambda)),\quad f_1((\lambda,1-\lambda))\quad \text{at second step;}\\
&f_{00}((\lambda,1-\lambda)),\quad f_{01}((\lambda,1-\lambda))\\
&f_{10}((\lambda,1-\lambda)),\quad f_{11}((\lambda,1-\lambda))\quad \text{at third step;}\\
\cdots\\
&f_{\sigma}((\lambda,1-\lambda)),\quad \sigma\in\set{0,1}^{n-1}\quad \text{at $n$-th step;}\\
\cdots.
%&(\lambda^3,\lambda^2(1-\lambda)),\quad(\lambda(\lambda^2+1-\lambda),\lambda(1-\lambda^2))\\
%&(\lambda^3+1-\lambda,\lambda^2(1-\lambda)+1-\lambda),\quad(\lambda(\lambda^2+1-\lambda)+1-\lambda,\lambda(1-\lambda^2)+1-\lambda)\quad \text{at third step}
\end{align*}
At $k$-th construction step, we remove $2^{k-1}$ open intervals, denote by $(a_{k,i},b_{k,i})$, $i=1,2\ldots,2^{k-1}$, then we can write
\begin{equation*}
F_\lambda=[0,1]\setminus \cup_{k=1}^{\infty} \cup_{i=1}^{2^{k-1}}(a_{k,i},b_{k,i}).
\end{equation*}
Clearly, the total removed length maintains a fixed proportion $\lambda$ at each step.
The set $F^2_{\lambda}$ can be similarly constructed using an analogous approach, the only change is that the proportion of the deleted intervals is no longer a fixed value.

\begin{lemma}\label{lemma1-1}
$F^2_{\lambda}=[0,1]\setminus \cup_{k=1}^{\infty} \cup_{i=1}^{2^{k-1}}(a_{k,i}^2,b_{k,i}^2)$.

\end{lemma}

\begin{proof}
Suppose that $y=x^2\in F^2_{\lambda}$, where $x\in F_\lambda$. By the structure of $ F_\lambda$, we see $x\notin \cup_{k=1}^{\infty} \cup_{i=1}^{2^{k-1}}(a_{k,i},b_{k,i})$, i.e., $x\notin(a_{k,i},b_{k,i})$ for any $k,i$.
Fix a gap interval $(a_{k,i},b_{k,i})$ arbitrarily, we have either $x<a_{k,i}$ or $x>b_{k,i}$, thus $y<a_{k,i}^2$ or $y>b_{k,i}^2$. So $y\notin (a_{k,i}^2,b_{k,i}^2)$ for any $k,i$, we obtain $F^2_{\lambda}\subset [0,1]\setminus \cup_{k=1}^{\infty} \cup_{i=1}^{2^{k-1}}(a^2_{k,i},b^2_{k,i})$.

Inversely, let $x^2\in [0,1]\setminus \cup_{k=1}^{\infty} \cup_{i=1}^{2^{k-1}}(a_{k,i}^2,b_{k,i}^2)$, it remains to prove that $x\in F_\lambda$ and then $x^2\in F^2_{\lambda}$. Clearly, for any $n\in\nn^+$ we have $x^2\in [0,1]\setminus \cup_{k=1}^{n}\cup_{i=1}^{j}(a_{k,i}^2,b_{k,i}^2)$. %where$j=\sum_{n=0}^{k-1}2^n$.
Then $x\in [0,1]\setminus \cup_{k=1}^{n}\cup_{i=1}^{j}(a_{k,i},b_{k,i})$, therefore $x\in F_\lambda$.

\end{proof}

%As we know $C_\lambda$ is a self-similar set, we show in the following that, however, $C_\lambda^2$ is not a self-similar set.
Although $F_\lambda$ is self-similar by construction, we prove that $F_\lambda^2$ does not preserve self-similarity.
\begin{lemma}\label{lemma1-2}
$F^2_{\lambda}$ is not a self-similar set.
\end{lemma}
\begin{proof}
Suppose that $F^2_{\lambda}$ is self-similar,  there exists a $ f(x)=rx+b$ with $0<|r|<1,b\in \RR$ such that
\begin{equation*}
  f(F^2_{\lambda})\subset F^2_{\lambda} \qtq{and} 1\in f(F^2_{\lambda}).
\end{equation*}
The above relations imply that the $conv(f(F^2_{\lambda}))=[b,b+r]=[b,1]$ or $conv(f(F^2_{\lambda}))=[b+r,b]=[b+r,1]$.

Let
\begin{equation*}
   c_n:=(1-\lambda^n)^2 - \left(1-(1-\lambda)\lambda^{n-1}\right)^2,\quad n\in \mathbb N^+
\end{equation*}
and
\begin{equation*}
   R:=\set{\frac{c_n}{1-2\lambda}:n=1,2,\cdots }.
\end{equation*}
We have $c_n$ is decreasing, because $c_n=(1-2\lambda)\lambda^{n-1}(2-\lambda^{n-1})$ and
\begin{equation*}
\frac{c_n}{c_{n+1}}=\frac{2-\lambda^{n-1}}{\lambda(2-\lambda^{n})}.
\end{equation*}
Since $0<\lambda<\frac{1}{2}$ we have
\begin{equation*}
\lambda(\lambda^{n}-\lambda^{n-2}-2)+2>0\Rightarrow \frac{2-\lambda^{n-1}}{\lambda(2-\lambda^{n})}>1.
\end{equation*}
So $\frac{c_n}{c_{n+1}}>1$. We split the proof into two cases.

Case 1. $|r|\notin R$. Then there exists a positive integer $n$ such that $\frac{c_{n+1}}{1-2\lambda}<|r|<\frac{c_{n}}{1-2\lambda}$. Thus
\begin{equation*}
diam f(F^2_{\lambda})>\frac{c_{n+1}}{1-2\lambda}=\lambda^{n}(2-\lambda^{n})=1^2-(1-\lambda^n)^2.
\end{equation*}
Since the largest element of $f(F^2_{\lambda})$ is $1$, so the smallest element of $f(F^2_{\lambda})$ is less than $(1-\lambda^n)^2$. Note that
\begin{equation*}
 (1-\lambda^{n}-(1-2\lambda)\lambda^{n-1},1-\lambda^{n})=\left(1-(1-\lambda)\lambda^{n-1}, 1-\lambda^n\right)
\end{equation*}
is in the gap intervals of $F_{\lambda}$, thus $\left((1-(1-\lambda)\lambda^{n-1})^2, (1-\lambda^n)^2\right)$ is in the gap intervals of $F^2_{\lambda}$ by Lemma \ref{lemma1-1}. It follows from $f(F^2_{\lambda})\subset F^2_{\lambda}$ and $\min f(F^2_{\lambda})< (1-\lambda^n)^2$ that $\left((1-(1-\lambda)\lambda^{n-1})^2, (1-\lambda^n)^2\right)$ is in the  gap intervals of $f(F^2_{\lambda})$.

We show that the largest length interval in gap intervals
 of $F^2_{\lambda}$ is $\left(\lambda^2, (1-\lambda)^2\right)$. Recall that $(a_{k,i},b_{k,i}),i=1,2\ldots,2^{k-1}$ the removed intervals at $k$ step in the construction of $F_\lambda$, where $a_{k,1}<a_{k,2}<\cdots<a_{k,2^{k-1}}$ and $b_{k,1}<b_{k,2}<\cdots<b_{k,2^{k-1}}$. Since $b_{k,i}-a_{k,i}=b_{k,j}-a_{k,j}$ for any $1\le i,j\le 2^{k-1}$ at $k$ step, then we have
 \begin{equation*}
   b_{k,1}^2-a_{k,1}^2<b_{k,2}^2-a_{k,2}^2<\cdots<b_{k,2^{k-1}}^2-a_{k,2^{k-1}}^2.
 \end{equation*}
 Therefore, at $k$ step the largest length interval in  gap intervals
 of $F^2_{\lambda}$ is
 \begin{equation*}
  (a_{k,2^{k-1}}^2,b_{k,2^{k-1}}^2)=\left((1-(1-\lambda)\lambda^{k-1})^2, (1-\lambda^k)^2\right)
 \end{equation*}
and $b_{k,2^{k-1}}^2- a_{k,2^{k-1}}^2=c_k$. Hence the largest length interval in gap intervals
 of $F^2_{\lambda}$ is $(a_{1,1}^2,b_{1,1}^2)=\left(\lambda^2, (1-\lambda)^2\right)$, because $c_k$ is decreasing.

Denote by $G$ the largest length interval in gap intervals of $ f(F^2_{\lambda})$, and by $|G|$ the length of interval $G$.
So $ |G|\ge c_n$.
Secondly, the length of the largest interval in gap intervals of $F^2_\lambda$ is $1-2\lambda$ and $|r|<\frac{c_n}{1-2\lambda}$, which means that $|G|\le (1-2\lambda)|r|<c_n$, which leads to a contradiction.

Case 2. $|r|=\frac{c_{n}}{1-2\lambda}\in R$ for some positive integer $n$. Since $1\in f(F^2_{\lambda})$ and $f(F^2_{\lambda})\subset F^2_{\lambda}\cap [(1-\lambda^{n-1})^2,1]$. Then we have
\begin{equation*}
 G=\left((1-\lambda^{n-1}+\lambda^{n})^2, (1-\lambda^n)^2\right),
\end{equation*}
and hence $1,(1-(1-\lambda)\lambda^{n-1})^2, (1-\lambda^n)^2\in f(F^2_\lambda)$. Since the largest length interval in gap intervals
 of $F^2_{\lambda}$ is $\left(\lambda^2, (1-\lambda)^2\right)$, thus if $r>0$
\begin{equation*}
 f(1)=1, \quad f(\lambda^2)=(1-\lambda^{n-1}+\lambda^{n})^2.%f((1-\lambda)^2)=(1-\lambda^{n})^2.
\end{equation*}
Then
\begin{align*}
&b=1-r,\quad \lambda^2 r+b=(1-\lambda^{n-1}+\lambda^{n})^2\\
\Rightarrow &  r(\lambda^{2}-1)+1=(1-\lambda^{n-1}+\lambda^{n})^2\\
\Rightarrow &  r=\frac{(1-\lambda^{n-1}+\lambda^{n})^2-1}{\lambda^{2}-1}=\frac{\lambda^{n-1}(2+\lambda^{n}-\lambda^{n-1})}{\lambda+1}\\
\Rightarrow & \lambda^{n-1}(2-\lambda^{n-1})=\frac{\lambda^{n-1}(2+\lambda^{n}-\lambda^{n-1})}{\lambda+1}.
\end{align*}
By the last equation $\lambda$ has no solution in $(0,1/2)$ when $n\ge2$, and $\lambda$ can take any value in $(0,1/2)$ when $n=1$. If $n=1$ then $r=1$, this is impossible.
%By the above equations $r=\frac{2\lambda^n-(1+\lambda)\lambda^{2n}}{1+\lambda}$, equivalently
%\begin{equation*}
%  (2-4\lambda+2\lambda^{n}-\lambda^{n-1})/(1-2\lambda)=(2-\lambda^{n-1}-\lambda^n)/(1+\lambda),
%\end{equation*}
%then we obtain $2\lambda-4\lambda^2=0$ by simplifying the above equation, so $\lambda=0$ or $\lambda=1/2$, this is impossible.
If $r<0$, similarly
\begin{equation*}
 f(0)=1, f((1-\lambda)^2)=(1-\lambda^{n-1}+\lambda^{n})^2,%f(\lambda^2)=(1-\lambda^{n})^2,f((1-\lambda)^2)=(1-(1-\lambda)\lambda^{n-1})^2.
\end{equation*}
this is also impossible.

\end{proof}

\section{Proof of main results }\label{S3}

%\indent First, we introduce some notations. Let $E=[0,1]$. For any $\left(i_{1}, \ldots, i_{n}\right) \in\{1,2\}^{n}$, we call $f_{i_{1}, \ldots, i_{n}}([0,1])=\left(f_{i_{1}} \circ \cdots \circ f_{i_{n}}\right)([0,1])$ a basic interval of rank $n$, which has length $\lambda^{n}$. Denote by $E_{n}$ the collection of all these basic intervals of rank $n$. Suppose $A$ and $B$ are th¡¾e left and right endpoints of some basic intervals in $E_{k}$ for some $k \geq 1$, respectively. Denote by $G_{n}$ the union of all the basic intervals of rank $n$ which are contained in $[A, B]$. Let $I$ be a basic interval with rank $n$. Define $\widetilde{I}=f_{1}(I) \cup f_{2}(I)$.
We first introduce some notations. We call $f_{\sigma}([0,1])=(f_{i_1}\circ\cdots \circ f_{i_n})([0,1])$ a basic interval of rank $n$, where $\sigma=i_1\cdots i_n\in\set{0,1}^n$.
Define
\begin{equation*}
\mathcal I_n:=\set{I: \text{$I$ is a basic interval of rank $n$}},
\end{equation*}
and for any $I\in \mathcal I_n$, let $I'$ be the union of all basic intervals of rank $n+1$ contained in $I$. Let $L$ and $R$ be the left and right endpoints of some basic intervals in $\mathcal I_k$ for some $k\geq 1$. Denote by $U_n$ the union of all the
basic intervals of rank $n$ which are contained in $[L,R]$. Now we introduce a technical lemma proved in \cite[lemma 2.1]{TJ}(see also \cite{ART}):
\begin{proposition}[\cite{TJ}]\label{lemma1}
 Suppose that $\phi: \mathbb{R}^2 \rightarrow \mathbb{R}$ is a continuous function. Let $L$ and $R$ be the left and right endpoints of some basic intervals in $\mathcal I_k$ for some $k\geq 1$. If $\phi(I, J)=\phi(I', J')$ for any $I, J \in \mathcal I_{n}$ and any $n \geq k$, then
\begin{equation*}
 \phi\left(F_\lambda \cap[L, R], F_\lambda \cap[L, R]\right)=\phi\left(U_k, U_k\right).
\end{equation*}

\end{proposition}
%Lemma 2.1. Let $F: U \rightarrow \mathbb{R}$ be a continuous function, where $U \subset \mathbb{R}^{2}$ is a nonempty open set. Suppose $A$ and $B$ are the left and right endpoints of some basic intervals in $G_{k_{0}}$ for some $k_{0} \geq 1$, respectively, such that $[A, B] \times[A, B] \subset U$. Then $K_{\lambda} \cap[A, B]=\bigcap_{n=k_{0}}^{\infty} G_{n}$. Moreover, if for any $n \geq k_{0}$ and any two basic intervals $I, J \subset G_{n}$,

%$$
%F(I, J)=F(\widetilde{I}, \widetilde{J})
%$$

%where $F(I, J):=\{F(x, y): x \in I, y \in J\}$, then

%$$
%F\left(K_{\lambda} \cap[A, B], K_{\lambda} \cap[A, B]\right)=F\left(G_{k_{0}}, G_{k_{0}}\right)
%$$

%\begin{lemma}\label{lemma1} Let $g(x, y) = \frac{x}{y}$, and let $ I = [a, a + t] $ and $ J = [b, b + t] $ be two basic intervals. If $\frac{1}{3} \leq \lambda < \frac{1}{2} $, and $ b \geq a \geq 1 - \lambda $, then $ g(\widetilde{I}, \widetilde{J}) = [R_1,S_1] \cup [R_2,S_2] \cup [R_3,S_3 \cup [R_4,S_4]$,
%where
%\begin{align*}
%&[R_1,S_1] = \left[ \frac{a}{b + t}, \frac{a + \lambda t}{b + t - \lambda t} \right],\\
%&[R_2,S_2]  = \left[ \frac{a}{b + \lambda t}, \frac{a + \lambda t}{b} \right],\\
%&[R_3,S_u3]  = \left[ \frac{a + t - \lambda t}{b + t}, \frac{a + t}{b + t - \lambda t} \right],\\
%&[R_4,S_4]  = \left[ \frac{a + t - \lambda t}{b + \lambda t}, \frac{a + t}{b} \right],\\
%\end{align*}
%and $S_1\ge R_2, S_2\ge R_3, R_3\ge R_2, S_3\ge R_4$(when $b\neq a$) and $S_2\ge R_4$(when $b=a$). Moreover, $g(\widetilde{I}, \widetilde{J}) = g(I, J)=[R_1,S_4]$.

%\end{lemma}

\begin{lemma}\label{lemma2} Let $ \Lambda(x, y) = \frac{x^2}{y}$ and $ J_1 = [u_1, u_1 + \lambda^j]  , J_2 = [u_2, u_2 + \lambda^j] \in \mathcal I_j$ for some $j\ge1$ with $u_2+\lambda \geq u_1+\lambda \ge 1$. If $\frac{1}{3} \leq \lambda < \frac{1}{2} $, we have $ \Lambda(J_1, J_2)= \Lambda(J_1', J_2')$.
\end{lemma}
\begin{proof} First, by the definition we obtain the offspring sets of $J_1,J_2$
\begin{align*}
&J_1' = [u_1, u_1 + \lambda^{j+1} ] \cup [u_1 + \lambda^{j} - \lambda^{j+1}, u_1 + \lambda^{j}],\\
&J_2' = [u_2, u_2 + \lambda^{j+1}] \cup [u_2 +  \lambda^{j}- \lambda^{j+1}, u_2 + \lambda^{j}],
\end{align*}
and then $\Lambda(J_1', J_2')$ is the union of the following sets:
\begin{align*}
&[l_1,r_1] = \left[ \frac{u_1^2}{u_2 +  \lambda^{j}}, \frac{(u_1 +  \lambda^{j+1})^2}{u_2 +  \lambda^{j} -  \lambda^{j+1}} \right],\\
&[l_2,r_2]  = \left[ \frac{u_1^2}{u_2 +  \lambda^{j+1}}, \frac{(u_1 +  \lambda^{j+1})^2}{u_2} \right],\\
&[l_3,r_3]  = \left[ \frac{(u_1 +  \lambda^{j} -  \lambda^{j+1})^2}{u_2 +  \lambda^{j}}, \frac{(u_1 +  \lambda^{j})^2}{u_2 +  \lambda^{j} - \lambda^{j+1}} \right],\\
&[l_4,r_4]  = \left[ \frac{(u_1 +  \lambda^{j} -  \lambda^{j+1})^2}{u_2 +  \lambda^{j+1}}, \frac{(u_1 +  \lambda^{j})^2}{u_2} \right].
\end{align*}
Clear, $l_1<l_2,r_1<r_2$, $l_3<l_4$ and $r_3<r_4$. Now we prove that $\Lambda(J_1', J_2')=\Lambda(J_1, J_2)$.
%We can deduce from Lemma \ref{lemma1} that if $\lambda\ge \frac{1}{3}$ then $s_1\ge r_2, r_3\ge r_2, s_3\ge r_4$(when $b\neq a$), because $a+t>a+(1-\lambda) t>a+\lambda t>a$.

%Note that \( g(I, J) = [r_1, s_4] \). In the following, we verify that \( g(I, J) = J_1 \cup J_2 \cup J_3 \cup J_4 \).

%Given that $b \geq a \geq 1 - \lambda $,
%$s_1\ge r_2,s_2\ge r_3,r_3\ge r_2,s_3\ge r_4$ are equivalent to
Since $t>(1-\lambda) t>\lambda t>0$, it can be verified that when $\frac{1}{3} \leq \lambda < \frac{1}{2}$
%then we have $r_1\ge l_2, l_3\ge l_2, r_3\ge l_4$, because $t>(1-\lambda) t>\lambda t>0$. Now we consider $r_2-l_3$
\begin{align*}
 &\frac{u_1 +\lambda^{j+1} }{u_2 + \lambda^{j}  - \lambda^{j+1} }\ge \frac{u_1}{u_2 + \lambda^{j+1} } \Rightarrow  r_1\ge l_2\\
 &\frac{u_1 + \lambda^{j}  - \lambda^{j+1} }{u_2 + \lambda^{j} }\ge \frac{u_1}{u_2 +\lambda^{j+1} } \Rightarrow
l_3\ge l_2\\
&\frac{u_1 + \lambda^{j} }{u_2 + \lambda^{j}  -\lambda^{j+1} }\ge  \frac{u_1 + \lambda^{j}  - \lambda^{j+1} }{u_2 + \lambda^{j+1} }  \Rightarrow   r_3\ge l_4.
\end{align*}
We consider $r_2-l_3$
\begin{align*}
 &\frac{(u_1 + \lambda^{j+1})^2}{u_2}- \frac{(u_1 +\lambda^{j} -\lambda^{j+1})^2}{u_2 + \lambda^{j}}\\
 &=\frac{\lambda^{j} \left[ u_1^2 + 2u_1 u_2 (2 \lambda - 1) + \left(2u_1 \lambda + u_2 (2 \lambda - 1)\right) \lambda^{j} +\lambda^{2j+2} \right]}{u_2(u_2 + \lambda^{j})}.
\end{align*}
If $u_1^2 + 2u_1 u_2 (2 \lambda - 1)\ge0$ and $2u_1 \lambda + u_2 (2 \lambda - 1)\ge0$, then
\begin{align*}
 &u_1 + 2 u_2 (2 \lambda - 1)\ge 0 \iff  \lambda\ge \frac{1}{2}-\frac{u_1}{4u_2}\\
 &2u_1 \lambda + u_2 (2 \lambda - 1)\ge0 \iff  \lambda\ge \frac{u_2}{2(u_1+u_2)}=\frac{1}{2(u_1/u_2+1)}.
\end{align*}
It follows from $ 1>u_2 \geq u_1 \ge 1-\lambda$ that $1- \lambda \le \frac{u_1}{u_2}\le 1$. So the inequalities hold: for $ \lambda\ge \frac{1}{3}$ we have $u_1 + 2 u_2 (2 \lambda - 1)\ge 0$; for $ \lambda\ge \frac{2-\sqrt{2}}{2}$ we obtain $2u_1 \lambda + u_2 (2 \lambda - 1)\ge0$. Consequently, when $\lambda\ge \frac{1}{3}$ the relation $r_2>l_3$ holds. Since $\Lambda(J_1, J_2)=[l_1,r_4]$, we deduce that $\Lambda(J_1', J_2')=[l_1,r_4]=\Lambda(J_1, J_2)$.

%so $\lambda\ge \eta_n$, where $\eta_n=\frac{a+t-a(1+\frac{t}{b})^\frac{1}{n}}{t[(1+\frac{t}{b})^\frac{1}{n}+1]}\in (\frac{b-a}{2b+t},\frac{1}{2})$ and $(\eta_n)$ is an increasing sequence, converging to $\frac{1}{2}$. If $b=a$, $s_3\ge r_4$ is equivalent to
%\begin{equation*}
%\frac{(a + t - \lambda t)^{n+1}}{(a +  t)^n} \le a +\lambda t.
%\end{equation*}
%Let $f(\lambda)=\frac{(a + t - \lambda t)^{n+1}}{(a +  t)^n}-a -\lambda t$, then $f(\lambda)$ is continuous and decreasing on $[0,\frac{1}{2}]$, and $f(\frac{1}{2})<0, f(0)>0$. By zero theorem there exists a smallest $\zeta_n$ such that $f(\lambda)<0$ for any $\zeta_n<\lambda<\frac{1}{2}$.

%Let $\lambda_*=\max\set{\frac{1}{3},\eta_n,\zeta_n}$, according to the reasoning above, we can conclude that if $\lambda\in(\lambda_*,\frac{1}{2})$ then $g(\widetilde{I}, \widetilde{J})=[r_1,s_4]$.

\end{proof}

\begin{lemma}\label{lemma3}
Let $D$  be defined in \eqref{D-1}. If $\frac{1}{3} \leq \lambda < \frac{1}{2}$ then
$
D=[0, +\infty).
$
\end{lemma}
\begin{proof}

%We recall that $C_{\lambda}$ is the invariant set generated by the IFS $\left\{f_{i}(x)=\lambda x+i(1-\lambda)\right\}_{i=0}^{1}$,
As an consequence of Proposition \ref{lemma1} and Lemma \ref{lemma2} is
\begin{equation*}
 \frac{f_1^2(F_{\lambda})}{f_1(F_{\lambda})} = \left[(1 - \lambda)^2, \frac{1}{1 - \lambda}\right]\qtq{for $\frac{1}{3}\le \lambda < \frac{1}{2}$.}
\end{equation*}
The self-similarity of $F_{\lambda}$ ensures the scaling relation for any nonzero $c\in F_{\lambda}$,  we have the form
%The structure of $C_{\lambda}$ implies that any
%$x \in C_{\lambda} \backslash\{0\}$ is of the form
\begin{equation*}
\frac{c}{c'} = \lambda^k \quad \text{for some $k \in \nn_0$ and some $c' \in f_1(F_{\lambda}$)}.
\end{equation*}
%Then for any  $x = \lambda^m x^*\in C_{\lambda} \backslash\{0\}$ and $y = \lambda^n y^* \in C_{\lambda} \backslash\{0\}$ where $m,n\in\nn$ and $x^*, y^* \in f_2(C_{\lambda})$, we have  the form
Consider any nonzero elements $c_1 = \lambda^n c_1'\in F_{\lambda}$ and $c_2 = \lambda^m c_2' \in F_{\lambda}$ where $m,n\in\nn_0$ and $c_1', c_2' \in f_1(F_{\lambda})$. These forms immediately yield the identity:
\begin{equation*}
 \frac{c_1^2}{c_2} = \lambda^{2n - m} \cdot \frac{(c_1')^2}{c_2'}.
\end{equation*}
Since $F_{\lambda} \backslash\{0\}=\cup_{i=0}^{\infty} \cup_{c'\in f_1(F_{\lambda})}\lambda^ic'$, we derive that
\begin{equation*}
 D= \{0\} \cup \bigcup_{i=-\infty}^{\infty}   \left[\lambda^i(1 - \lambda)^2, \frac{\lambda^i}{1 - \lambda}\right].
\end{equation*}
 %the intervals $\lambda^i \left[(1 - \lambda)^2, \frac{1}{1 - \lambda}\right]$ are pairwise disjoint when $\frac{1}{3} \leq \lambda < \alpha\approx 0.317$, and
Hence, we establish the result $D=[0,+\infty)$ for all $ \lambda\ge \alpha\approx 0.318 $, where $\alpha$ denotes the unique real root of the equation $x^3-3x^2+4x-1=0$.
\end{proof}

%Suppose $F$ is a self-similar set, the ratio is $r$ and the largest gap interval is $G$. \cite[Theorem 1.5]{LJ} studied the structure of the image set of self-simlar set under continuous mapping:
Let $F_1,F_2$ be the attractors generated by
\begin{equation*}
\set{H_i(x)=\lambda x+\delta_i}_{i=0}^{n}, \set{T_j(x)=\lambda x+\eta_j}_{j=0}^{m},
\end{equation*}
where $0<\lambda<1,\delta_i,\eta_j\in\mathbb R$. We assume that the convex hulls of $F_1,F_1$ are $[A_1,B_1],[A_2,B_2]$ respectively, and $H_0(A_1)=A_1,H_{n}(B_1)=B_1,T_0(A_2)=A_2,T_{m}(B_2)=B_2$.
Denote by $G_1,G_2$ the maximal length intervals in gap intervals of $F_1,F_2$ respectively. We ignore the trivial cases where $F_1$ or $F_2$ is an interval. If $F_1,F_2$ are homogeneous Cantor sets then
\begin{equation*}
 (H_i(B_1),H_{i+1}(A_1))=G_1,(T_j(B_2),T_{j+1}(A_2))=G_2.
\end{equation*}
We denote the length of
an interval $J$ by $|J|$. In \cite[Theorem 1.5]{LJ}, Li et al. studied the structure of the image set of homogeneous self-similar sets under continuous mapping.
\begin{proposition}[\cite{LJ}]\label{pro1}
Let $F_1,F_2$ be the attractors defined as above, and their convex hull is $[0,1]$. Suppose $ g(x,y)\in C^3 $, if for any $(x,y) \in [0,1] \times [0,1],0\le i\le n-1,0\le j\le m-1$, we have
\begin{equation*}
\begin{cases}
\partial_x g(x,y) \neq 0, \partial_y g(x,y) \neq 0 \\
(c_{xy} \lambda  \partial_x + (H_i(1)-H_{i+1}(0)) \partial_y)^2 g(x,y) \geq 0 \\
(c_{xy}(T_j(1)-T_{j+1}(0)) \partial_x + \partial_y)^2 g(x,y) \geq 0\\
|G_1| \leq c_{xy}\frac{\partial_y g(x,y)}{\partial_x g(x,y)} \leq \frac{\lambda}{|G_2|},
\end{cases}
\end{equation*}
 where $(\ell_1\partial_x + \ell_2\partial_y)^2 g(x,y) = \ell_1^2 \partial_{xx} g(x,y) + 2\ell_1\ell_2 \partial_{xy} g(x,y) + \ell_2^2 \partial_{yy} g(x,y)$, and
\begin{equation*}
c_{xy}=
\begin{cases}
1, \text{if $\partial_x g(x,y)\partial_y g(x,y)>0$ for any $(x,y)\in[0,1]^2$} \\
-1, \text{if $\partial_x g(x,y)\partial_y g(x,y)<0$ for any $(x,y)\in[0,1]^2$}.
\end{cases}
\end{equation*}
Then $g(F_1, F_2)$ is a closed interval.%is a union of finitely many closed intervals.
\end{proposition}
Given $0<\lambda<\frac{1}{2}$, suppose that $F'_\lambda$ is the attractor generated by
\begin{equation*}
\left\{h_{0}(x)=\lambda x+(1-\lambda)^2, h_{1}(x)=\lambda x+1-\lambda\right\}, %\quad 0<\lambda<1 / 2
\end{equation*}
and then
\begin{equation}\label{define}
 F'_{\lambda}=h_{0}\left(F'_{\lambda}\right) \cup h_{1}\left(F'_{\lambda}\right).
\end{equation}
We see $F'_{\lambda}=F_{\lambda}\cap[1-\lambda,1]$ and $F'_{\lambda}=f_1(F_{\lambda})$, so $F'_\lambda$ is a subset of $ F_\lambda$.
The maximal length of interval in gap interval of $F'_\lambda$ is $|G|=\lambda - 2\lambda^2$, the convex hull of $F'_\lambda$ is interval $[1-\lambda,1]$, and $h_0(1-\lambda)=1-\lambda,h_1(1)=1$.

Based on Proposition \ref{pro1}, we obtain the following corollary:
%studied the structure of the image set of self-simlar set under continuous mapping:

\begin{lemma}\label{lemma4}
Let $F'_{\lambda}$ be defined in \eqref{define}. Suppose $ g(x,y)=\frac{x^2}{y} $, if for any $(x,y) \in [1-\lambda,1] \times [1-\lambda,1]$, the conditions
\begin{equation*}
\begin{cases}
%\partial_x g(x,y) \neq 0, \partial_y g(x,y) \neq 0 \\
(-\lambda  \partial_x - (\lambda - 2\lambda^2) \partial_y)^2 g(x,y) \geq 0 \\
((\lambda - 2\lambda^2) \partial_x + \partial_y)^2 g(x,y) \geq 0\\
\lambda - 2\lambda^2 \leq -\frac{\partial_y g(x,y)}{\partial_x g(x,y)} \leq \frac{1}{1 - 2\lambda}
\end{cases}
\end{equation*}
hold. Then $g(F'_{\lambda}, F'_{\lambda})$ is a closed interval.%is a union of finitely many closed intervals.
\end{lemma}

\begin{proof}
Let $F_1=F_2=F'_{\lambda}$ and $(x,y) \in [1-\lambda,1]^2$. It follows from $\partial_x g(x,y)=\frac{2x}{y}>0, \partial_y g(x,y)=\frac{-x^2}{y^2}<0$ that $c_{xy}=-1$. On the other hand,
\begin{equation*}
  h_{1}(1-\lambda)-h_0(1)=H_1(1-\lambda)-H_{0}(1)=T_1(1-\lambda)-T_{0}(1),
\end{equation*}
and $h_{1}(1-\lambda)-h_0(1)=|G_1|=|G_2|=\lambda - 2\lambda^2$. Therefore we can conclude by Proposition \ref{pro1}.
\end{proof}

We deduce from Lemma \ref{lemma4}:
%\begin{lemma}\label{lemma-3}
%Suppose $ g(x,y)=\frac{x^2}{y} $, if for any $(x,y) \in [1-\lambda,1]\times[1-\lambda,1]$, the conditions
%\begin{equation}\label{condition}
%\begin{cases}
%\partial_x g(x,y) \neq 0, \partial_y g(x,y) \neq 0 \\
%(-\lambda^2  \partial_x + ( 2\lambda^2-\lambda) \partial_y)^2 g(x,y) \geq 0 \\
%((\lambda - 2\lambda^2) \partial_x + \partial_y)^2 g(x,y) \geq 0\\
%\lambda - 2\lambda^2 \leq -\frac{\partial_y g(x,y)}{\partial_x g(x,y)} \leq \frac{\lambda}{1 - 2\lambda}
%\end{cases}
%\end{equation}
%hold, then $f(C'_\lambda, C'_\lambda)$ is a closed interval.%is a union of finitely many closed intervals.
%\end{lemma}

\begin{lemma}\label{lemma5}
$\frac{F'^2_{\lambda}}{F'_{\lambda} }$ is a closed interval for $\frac{2-\sqrt{2}}{2}\le  \lambda <\frac{1}{2}$.
\end{lemma}

\begin{proof}
Let $g(x,y)=\frac{x^2}{y}$ and $(x,y) \in [1-\lambda,1]^2$.
%The result established in Lemma \ref{lemma4} is still valid in this scenario, because we can translate $F$ to obtain a new set, it leaves invariant both the contraction ratio $r$ and the maximal gap length $|G|$.
%The result of Proposition \ref{lemma4} is still valid for this case, because we can translate $F$ to obtain a new set, the ratio $r$ and $|G|$ do not change.
%It suffices to prove that the following conditions are satisfied when $0< \lambda <\frac{1}{2}$.
In terms of Lemma \ref{lemma4} it remains to verify that the following conditions hold for all $\frac{2-\sqrt{2}}{2}\le  \lambda <\frac{1}{2}$
\begin{equation}\label{condition}
\begin{cases}
%\partial_x g(x,y) \neq 0, \partial_y g(x,y) \neq 0 \\
(-\lambda  \partial_x + ( 2\lambda^2-\lambda) \partial_y)^2 g(x,y) \geq 0 \\
((\lambda - 2\lambda^2) \partial_x + \partial_y)^2 g(x,y) \geq 0\\
\lambda - 2\lambda^2 \leq -\frac{\partial_y g(x,y)}{\partial_x g(x,y)} \leq \frac{1}{1 - 2\lambda}.
\end{cases}
\end{equation}
Note that
\begin{align*}
  &\partial_x g(x,y)=\frac{2x}{y} , \partial_y g(x,y)=\frac{-x^2}{y^2},  \\
  &\partial_{xx} g(x,y)=\frac{2}{y}, \partial_{xy} g(x,y)=\frac{-2x}{y^2},
  \partial_{yy} g(x,y)=\frac{2x^2}{y^3}.
\end{align*}
Then
\begin{align*}
 &(-\lambda  \partial_x + ( 2\lambda^2-\lambda) \partial_y)^2 g(x,y)\\
 =&\lambda^2 \partial_{xx}g(x,y)-2\lambda( 2\lambda^2-\lambda)\partial_{xy} g(x,y)+( 2\lambda^2-\lambda)^2 \partial_{yy} g(x,y)\\
=&\frac{2}{y}\left[\lambda^2 -2\lambda(\lambda- 2\lambda^2)\frac{x}{y}+( 2\lambda^2-\lambda)^2 \frac{x^2}{y^2}\right]=\frac{2}{y}\left[\lambda-(\lambda- 2\lambda^2)\frac{x}{y}\right]^2\ge0.\\
&((\lambda- 2\lambda^2) \partial_x + \partial_y)^2 g(x,y)\\
 =&(\lambda- 2\lambda^2)^2 \partial_{xx}g(x,y)+2( \lambda- 2\lambda^2)\partial_{xy} g(x,y)+ \partial_{yy} g(x,y)\\
=&\frac{2}{y}\left[(\lambda- 2\lambda^2)^2 -2( \lambda- 2\lambda^2)\frac{x}{y}+ \frac{x^2}{y^2}\right]=\frac{2}{y}\left(\lambda- 2\lambda^2-\frac{x}{y}\right)^2\ge0.
\end{align*}
Finally, $-\frac{\partial_y g(x,y)}{\partial_x g(x,y)}=\frac{x}{2y}\in[\frac{1-\lambda}{2},\frac{1}{2(1-\lambda)}]$ for all $x,y\in[1-\lambda,1]$.  If
\begin{equation}\label{condition-1}
\frac{1}{2(1-\lambda)}\leq \frac{1}{1 - 2\lambda} \qtq{and} \frac{1-\lambda}{2} \geq \lambda - 2\lambda^2,
\end{equation}
then we have $\lambda - 2\lambda^2 \leq -\frac{\partial_y g(x,y)}{\partial_x g(x,y)} \leq \frac{1}{1 - 2\lambda}$ for all $x,y\in[1-\lambda,1]$. The inequalities in \eqref{condition-1} hold while $\lambda \ge \frac{2-\sqrt{2}}{2}\approx 0.293$. So \eqref{condition} is satisfied.

\end{proof}

We give the proof of main theorem.
\begin{proof}[Proof of Theorem \ref{theorem-1}]  We can obtain Theorem \ref{theorem-1} (i)  immediately follows from Lemma \ref{lemma1-2}.
(ii) was established in Lemma \ref{lemma3}. For (iii), recall that $F'_\lambda$ defined in \eqref{define} is a scaled copy of $F_\lambda$, satisfying $F'_\lambda \subset F_\lambda$. This inclusion immediately yields $\frac{F'^2_{\lambda}}{F'_{\lambda} }\subset \frac{F^2_{\lambda}}{F_{\lambda} }$. The conclusion then follows by applying Lemma \ref{lemma5}.

\end{proof}

\section{Final remarks }\label{S4}
In our work, we investigate the structure of $F^2_{\lambda}$ and the visibility of set $F_{\lambda}\times F^2_{\lambda}$. We may consider the sets $F^n_{\lambda}$ and $F_{\lambda}\times F^n_{\lambda},n\ge3$. For this extension the proof is more complicated, and the main
difficulty is to find out the range of $\lambda$ such that $ \phi(J_1, J_2)= \phi(J_1', J_2')$, where $\phi(x,y)=\frac{x^n}{y},J_1,J_2\in \mathcal I_j$. Finally, it is natural to ask can we give some characterizations of  $\set{\Phi:\text{$\Phi(F_{\lambda})$ is not self-similar}}$.

\vspace{1cm}

%\section{Conclusion}
%In this work, we investigate the visibility of non-self-similar sets. We begin by analyzing the structure of $F^2_\lambda$, we show that it lacks self-similarity. Due to the nonlinear phenomena exhibited by $F^2_\lambda$, we develop a different approach to characterize the visible set.
%combining methods from fractal theory, numerical computation, and dynamical systems theory.
%Our results reveal that the visible set may contain a closed interval, and its structure differs from that of self-similar sets. In the future, we will explore visibility problem of more general non self-similar sets.

%In this work, we observe the visibility on non-self-similar sets. First we give a description of $C^2_\lambda$, we find that it is not self similar set. Since $C^2_\lambda$ presents a nonlinear phenomenon, so we use different way to investigate the visible set. We apply fractal theory, numerical calculation, and dynamical system theory to solve problems. We obtain the visible set may contain a closed interval, and we find that the structure of visible set of non self-similar sets are different from that ofself-similar sets. In the future, we will explore visibility problem of more general non self-similar sets.

%{\bf Declaration of competing interest}
%The authors declare that they have no known competing financial
%interests or personal relationships that could have appeared to influence
%the work reported in this paper.

{\bf Acknowledgments}
%The authors would like to thank to Dr. Zhiqiang Wang for his useful comments and suggestions.

%The authors would like to thank to Dr. Zhiqiang Wang for his useful comments and suggestions.
The authors are grateful to Dr. Zhiqiang Wang for valuable discussions on our paper, and also thank the referees for their useful suggestions and comments. The first author was supported by National Natural Science Foundation of China (grant No. 12301108).

%{\bf Data availability}
%No data was used for the research described in the article.

 \end{document}